\documentclass[11pt]{article}
\usepackage{geometry}                % See geometry.pdf to learn the layout options. There are lots.
\usepackage{graphicx}
\usepackage{amssymb}
\usepackage{epstopdf}
\usepackage{amssymb}
\usepackage{amsthm}
\usepackage{amsmath}

\usepackage{hyperref}

\title{Algorithms for computing maximal lattices in bilinear \\ (and quadratic) spaces over number fields}
\author{Jonathan Hanke}
\date{\today}                                           % Activate to display a given date or no date

\begin{document}
\maketitle

\newtheorem{thm}{Theorem}[section]  
\newtheorem{lem}[thm]{Lemma}  
\newtheorem{cor}[thm]{Corollary}  
\newtheorem{defn}[thm]{Definition}  
\newtheorem{rem}[thm]{Remark}  
\newtheorem{conj}[thm]{Conjecture}
\newtheorem{alg}[thm]{Algorithm}  
\newtheorem{prob}[thm]{Problem}

\newcommand{\GCD}{\mathrm{GCD}}
\newcommand{\GL}{\mathrm{GL}}
\newcommand{\Char}{\mathrm{Char}}
\newcommand{\Rad}{\mathrm{Rad}}

\newcommand{\Ker}{\mathrm{Ker}}
\renewcommand{\Im}{\mathrm{Im}}

\newcommand{\NormF}{\mathrm{N}_{F/\Q}}
\newcommand{\Hom}{\mathrm{Hom}}

\newcommand{\even}{\text{-even}}
\newcommand{\aevenatp}{\a\even\text{ at $\p$}}

\newcommand{\B}{\mathcal{B}}

\newcommand{\x}{\vec{x}}
\newcommand{\y}{\vec{y}}
\renewcommand{\v}{\vec{v}}
\newcommand{\w}{\vec{w}}

\newcommand{\vr}{\vec{r}}
\newcommand{\vl}{\vec{l}}
\newcommand{\va}{\vec{a}}
\newcommand{\vi}{\vec{\iota}}

\newcommand{\Z}{\mathbb{Z}}
\newcommand{\N}{\mathbb{N}}
\newcommand{\Q}{\mathbb{Q}}
\newcommand{\F}{\mathbb{F}}
\renewcommand{\P}{\mathbb{P}}

\renewcommand{\SS}{\mathbb{S}}

\renewcommand{\O}{\mathcal{O}}
\newcommand{\C}{\mathcal{Q}}
\newcommand{\D}{\mathcal{D}}    %% The Discriminant Module

\renewcommand{\a}{\mathfrak{a}}
\renewcommand{\b}{\mathfrak{b}}
\renewcommand{\c}{\mathfrak{c}}

\newcommand{\p}{\mathfrak{p}}
\renewcommand{\l}{\mathfrak{l}}
\newcommand{\n}{\mathfrak{n}}
\newcommand{\s}{\mathfrak{s}}

\newcommand{\al}{\alpha}
\newcommand{\ve}{\varepsilon}
\newcommand{\ord}{\mathrm{ord}}

\newcommand{\ra}{\rightarrow}
\newcommand{\surj}{\twoheadrightarrow}

\newcommand{\half}{\tfrac{1}{2}}

\renewcommand{\L}{\errorrrrrr}

\newcommand{\latI}{\mathcal{I}}
\newcommand{\latK}{\mathcal{K}}
\newcommand{\latL}{\mathcal{L}}
\newcommand{\latM}{\mathcal{M}}
\newcommand{\latN}{\mathcal{N}}

\newcommand{\Span}{\text{Span}}

\newcommand{\pip}{\pi_\p}
\newcommand{\q}{{\mathfrak q}}

\begin{abstract}
In this paper we describe an algorithm that quickly computes a maximal $\a$-valued lattice in an $F$-vector space equipped with  a non-degenerate bilinear 
form, where $\a$ is a fractional ideal in a number field $F$.  
We then apply this construction to give an algorithm to compute an $\a$-maximal lattice in a quadratic space over any number field $F$ where the prime $p=2$ is unramified.
We also develop the theory of $\p$-neighbors for 
$\a$-valued quadratic lattices at 
an arbitrary prime $\p$ of $\O_F$ (including when $\p\mid 2$) and prove its close connection to the residual geometry of certain quadrics mod $\p$.
Finally we give a well-known application of $\p$-neighboring lattices and exact mass formulas 
to compute a complete set of representatives for the classes in a given genus of (totally definite) quadratic $\O_F$-lattices.
\end{abstract}

%\footnotetext{Mathematics Subject Classification: 11E12, 11E39, 11E41}

%%%%%%%%%%%%%%%%
%%  SECTION: Introduction %%
%%%%%%%%%%%%%%%%

\section{Introduction and Notation}
In the study of the arithmetic of quadratic forms there has historically been a strong focus on explicit computations and specific examples.  One 
of the first and most important instances
of this was Gauss's computation of (proper integral) equivalence classes of primitive positive definite binary quadratic forms of fixed discriminant, and their arrangement into genera based on the values of certain ``genus characters''. These numerical investigations led to important conjectures about quadratic fields and quadratic forms of class number one that have only recently begun to be resolved.  (See \cite{Stark__Gauss_class_number_problem, Goldfeld__Gauss_class_number_problem} for an overview.)  

Another well-known example is the explicit formula of Jacobi for the number $r_4(m)$ of representations of a positive integer $m$ as a sum of four integer squares, given by
$$
r_4(m) = 8 \sum_{\substack{0<d\mid m \\ 4\nmid d}} d > 0,
$$
and the many subsequent efforts of other authors 
to prove similar explicit representability and representation number formulas for other positive definite quadratic forms.

In recent times, the use of computers offers us the potential to perform previously unimaginable computations that can extend both the scope of our vision and our ability to prove concrete enumerative theorems 
too complex for a more traditional ``by-hand'' case-by-case enumeration.  In the arithmetic theory of quadratic forms, this is still a promise 
largely waiting to be realized (e.g. see Remark \ref{rem:class_enumeration_literature}).

One of the fundamental objects in this theory is the 
{\bf maximal (integer-valued) quadratic lattice}, both because these have the fewest complications at ``bad'' primes (i.e. primes dividing the level of the associated local quadratic forms), and because there is exactly one genus of maximal quadratic lattices in any (non-degenerate) quadratic space.
These are very much analogous to studying maximal orders in number fields, or more generally in central simple algebras over them, and many theorems become substantially simpler in that context (e.g. \cite{Shimura_exact_mass, GHY_mass, Bocherer-Nebe, Walling__Explcit_Siegel_Theory, Shimura_Clifford_groups_book}).

While the importance of maximal lattices in the theory has been clear for a long time (e.g. \cite{Brandt1932, Brandt1937, Brandt1945}), proving explicit enumerative results 
even in this simplified context 
has been a rather daunting endeavor due to their complexity and many opportunities for errors.  A pioneer in these investigations has been Shimura, whose many papers \cite{Shimura__Mass_survey_article, Shimura_exact_mass, Shimura__number_of_representations, Shimura__mass_formulas_and_ternary_class_number_one_enumeration} and recent book \cite{Shimura_Clifford_groups_book} focusing on the arithmetic of maximal lattices have set the stage for other authors' work \cite{Ha_thesis_paper, GHY_mass, Yoshinaga2010, HIraka2006, Murata2007}.  Several other papers in a different style where maximal lattices play an important role are \cite{Beliopetsky2007, Ponomarev1976, Ponomarev1981, Brzezinski1974, Yang2004}, and they are also mentioned in the introductory books \cite[\S82H and \S104:9-10]{OMeara:1971zr} and \cite[\S9.3]{Gerstein_book}. 

Our hope is that this paper and the supporting open-source implementation \cite{Ha-Sage-QF-class, Hanke:uq} over $\Q$ in the freely-available Sage computer algebra system \cite{sage-4.6.2} will make the arithmetic of maximal lattices more accessible to study and numerical experimentation.  
Among other things, this implementation includes functionality for computing with quadratic forms/spaces/lattices, as well as finding $p$-neighbors, genus representatives and maximal lattices when $F= \Q$.
One application of these algorithms is the author's recent work \cite{Hanke:CN1_maximal_over_QQ} enumerating all maximal definite quadratic lattices over $\Z$ of class number one in $n\geq 3$ variables.

\bigskip 

\noindent
{\bf Outline:}  The main results of this paper are to:
\begin{enumerate}
\item Prove an algorithm for computing a maximal bilinear lattice in a given non-degenerate bilinear space over an arbitrary number field.
\item Prove an algorithm for computing a maximal quadratic lattice in a given non-degenerate quadratic space over a number field where $p=2$ is unramified.
\item Explain a generalization of the theory of $\p$-neighbors allowing $\a$-valued quadratic lattices and its connection to residual geometry.
\item Prove an algorithm for enumerating the classes in a given genus of $\a$-valued quadratic lattices over an arbitrary number field.
\end{enumerate}

\bigskip 

\noindent
{\bf Acknowledgements:}  
This is an outgrowth of some Sage code written for quadratic forms over $\Z$ together with Gabrielle Nebe shortly after the Sage Days 13 Conference on ``Quadratic Forms and Lattices" held at the University of Georgia in March 2009.  The author gratefully acknowledges her guidance and helpful comments, as well as the thoughtful comments and corrections of the anonymous referee, and several references provided by Wai Kiu Chan and Rainer Schulze-Pillot. The author also warmly acknowledges the hospitality of the Boston University Mathematics Department for allowing him to visit for two months in Spring 2012, where the final writing of this paper occurred.  
This work was partially supported by the author's NSF Grant DMS-0603976.

\bigskip 

\noindent
{\bf Notation:}

Throughout this paper we denote by $\N, \Z, \Q, \mathbb{R}, \mathbb{C},$ and $\F_q$ respectively the positive integers, integers, rational numbers, real numbers, complex numbers, and finite field with $q$ elements.  If we take $F$ to be a number field (i.e. a finite dimensional field extension of $\Q$), then we let $\O_F$ be the ring of integers in $F$, $\p$ any non-zero prime ideal of $\O_F$, and take $\F_\p := \O_F/\p$ to be the (finite) residue field having $\NormF(\p)$ elements (where $\NormF(\p) \in \N$ is the absolute norm of $\p$).  We also usually let $\a$ and $\b$ denote (non-zero) fractional ideals of $F$ (i.e. invertible rank 1 $\O_F$-modules).

Given a number field $F$ and a (non-zero) prime ideal $\p$ of $\O_F$, we let $F_\p$ denote the $\p$-adic completion of $F$ and denote its (valuation) ring of integers by $\O_\p$.   By abuse of notation, we also denote the maximal ideal of $\O_\p$ as $\p$ and leave the reader to decide if the set $\p$ is $\p$-adically complete based on its usage.

If $V$ is a finite dimensional vector space over a 
field $K$ with ring of integers $R$, we say that a subset $L \subset V$ is a {\bf lattice (in V)} if $L$ is a finitely generated $R$-module whose $K$-span $L\otimes_R K = V$. Given a symmetric bilinear form $B:V \times V \rightarrow K$ we refer to the pair $(V,B)$ as a {\bf (symmetric) bilinear space} over $K$.  Similarly, given a quadratic form $Q:V \rightarrow K$ we refer to the pair $(V,Q)$ as a {\bf quadratic space} over $K$.  We also refer to an $\O_K$-lattice in a bilinear space or quadratic space over a number field $K$ respectively as a {\bf (global) bilinear} or {\bf (global) quadratic lattice}.  

Given a number field $K$ and a (non-zero) prime $\p$ of $\O_K$, we denote the associated {\bf local bilinear} or {\bf local quadratic spaces} at $\p$ respectively as $(V_\p, B_\p)$ or $(V_\p, Q_\p)$, where $V_\p := V\otimes_K K_\p$ and where $B_\p$ and $Q_\p$ denote the unique continuous extensions of $B$ and $Q$ to $V_\p$.  (We could also view $B_\p$ as the $K_\p$-linear extension of $B$ to $V_\p$, and $Q_\p$ as the $K_\p$-quadratic extension of $Q$ to $V_\p$.)  We also define the  (possibly bilinear or quadratic) {\bf local lattice} $L_\p := L \otimes_{\O_K} \O_p \subset V_\p$ associated to $L$ at $\p$.

Given a (local/global) bilinear space $(V,B)$, we define the {\bf associated quadratic space} $(V,Q_B)$ by $Q_B(\x) := B(\x,\x)$.  Similarly, given a quadratic space $(V,Q)$ we define the {\bf associated (Hessian) bilinear space} $(V,H)$ by $H(\x,\y) := H_Q(\x,\y) := Q(\x+\y) - Q(\x) - Q(\y)$.  Notice that these operations are {\it not inverses} of each other, and composing them has the effect of multiplying the (quadratic/bilinear) form by two.  (While it is often a convention to associate the ``Gram bilinear form'' $\frac12 H$ to a quadratic form, this is not natural unless 2 is a unit, and not even possible over fields of characteristic 2.)
We also adopt the convention that any notion for a bilinear space applied to a quadratic space $(V,Q)$ is applied to its associated bilinear space $(V,H)$, and 
conversely quadratic notions on a bilinear space $(V,B)$ are applied to the associated quadratic space $(V,Q)$.

Given a bilinear space $(V,B)$, we say that $\x, \y \in V$ are {\bf perpendicular} (or {\bf orthogonal})  $\iff B(\x,\y) = 0$, and for any subset $W\subseteq V$ we define $W^\perp :=\{\x \in V \mid B(W, \x) = \{0\} \}$.  We define the {\bf radical} of $V$ by $\mathrm{Rad}(V) := V^\perp$ and say that $V$ is {\bf non-degenerate} if $\mathrm{Rad}(V) = \{\vec 0\}$.  We call the 2-dimensional bilinear space $(V,B)$ given by $B((x_1, x_2),(y_1, y_2)) := x_1y_2 + x_2 y_1$ the {\bf hyperbolic plane}, and refer to any direct sum of these as a {\bf hyperbolic space}.  We say that a subspace $W \subseteq (V,B)$ is {\bf Lagrangian} if $W^\perp = W$, and that $W$ is {\bf weakly metabolic} if it admits a Lagrangian subspace.

We say that a vector $\v \neq \vec 0$ in a quadratic space $(V,Q)$ is {\bf isotropic} if $Q(\v) = 0$.
We say that a quadratic space is {\bf isotropic} if it contains an isotropic vector, and that is {\bf anisotropic} otherwise (i.e. $Q(\x)=0 \implies \x = \vec 0$).
We also refer to a subspace $W \subseteq (V,Q)$ as {\bf totally isotropic} if $Q(W) = \{0\}$.

Finally, given any vector space $V$ we define $\P(V)$ as the set of lines in $V$ passing through the origin $\vec 0$, and for any subset $\SS \subseteq V$ we define $\P(\SS) \subseteq \P(V)$ as the set of lines in $\P(V)$ having non-empty intersection with $\SS$.  We also define a {\bf maximal} object as being maximal with respect to the natural inclusion of such objects.  
For any $x \in \mathbb{R}$ we let $\lceil x\rceil$ denote the smallest integer $\geq x$ (i.e. the ceiling function).
For any ring $R$ 
we let $\Char(R)$ denote the characteristic of $R$, and   
for any subset $\SS$ of an $R$-module we let $\Span_R(\SS)$ as the $R$-module generated by $\SS$.

%%%%%%%%%%%%%%%%%%%%%%%%
%% SECTION:  Duality for  Bilinear Lattices  %%
%%%%%%%%%%%%%%%%%%%%%%%%
%%%%%%%%%%%%%%%%%%%%%%%%
%% SECTION:  Duality for  Bilinear Lattices  %%
%%%%%%%%%%%%%%%%%%%%%%%%
\section{Duality for Bilinear Lattices}

\subsection{Dual lattices and Modular lattices}

We begin by recalling some useful facts about bilinear forms and duality that we will use freely throughout the paper, as well as the all-important polarization identity.

\begin{defn}
Given a lattice $L$ in a bilinear space $(V, B)$ over a number field $F$, and a non-zero fractional ideal $\a$ of $F$, we define the 
{\bf $\a$-dual lattice} $L^{\#\a} \subset (V,B)$ of $L$ as the $\O_F$-lattice 
$$
L^{\#\a} := \{\v \in V\mid B(\v, L) \subseteq \a \}.
$$
When $\a= \O_F$ this is simply referred to as the {\bf dual lattice} of $L$, denoted by $L^\#$.
One of the main uses for the $\a$-dual lattice is to help understand the $\a$-valued superlattices of $L$.  
\end{defn}

\begin{lem}[$\a$-valued lattices and duals]
If $L$ is a lattice in a non-degenerate bilinear space $(V, B)$ over a number field $F$ and $\a$ is a non-zero fractional ideal of $F$, then 
$$
\text{$L$ is $\a$-valued}
\iff 
L \subseteq L^{\#\a}.
$$
\end{lem}

\begin{proof}
This follows since $L$ is $\a$-valued $\iff B(L,L) \subseteq \a \iff 
L \subseteq L^{\#\a}$.
\end{proof}

\begin{lem}[Lattice scaling and Duality]  \label{lem:lattice_scaling_and_duality}
Suppose that $\a$ and $\b$ are non-zero fractional ideals of $F$, and $L$ is a lattice in a non-degenerate bilinear space over $F$.  Then
$L^{\#\a} = \a L^{\#}$ and $(\b L)^\# = L^{\#\b^{-1}}$.
\end{lem}

\begin{proof}
Since $B(\a L^\#, L) \subseteq \a B(L^\#, L) \subseteq \a$ we know that $\a L^\# \subseteq L^{\#\a}$, and conversely $B(L^{\#\a}, L) \subseteq \a \implies B(\a^{-1}L^{\#\a}, L) \subseteq \O_F \implies \a^{-1} L^{\#\a} \subseteq L^{\#} \implies  L^{\#\a} \subseteq \a L^{\#}$, so we have the equality $L^{\#\a} = \a L^{\#}$.

Similarly $B((\b L)^\#, \b L) \subseteq \O_F \implies B((\b L)^\#, L)\subseteq \b^{-1} \implies (\b L)^\# \subseteq L^{\#\b^{-1}}$, and also $B(L^{\#\b^{-1}}, L) \subseteq \b^{-1} \implies B(L^{\#\b^{-1}}, \b L) \subseteq \O_F \implies  L^{\#\b^{-1}} \subseteq (\b L)^\#$, giving the desired equality $(\b L)^\# = L^{\#\b^{-1}}$.
\end{proof}

\begin{lem}[Double duals] \label{Lem:double_dual}
If $L$ is a bilinear lattice in a non-degenerate bilinear space $(V,B)$ and $\a$ is a non-zero fractional ideal of $F$, then 
$(L^{\#\a})^{\#\a} = L$.
\end{lem}

\begin{proof}
The equality $(L^{\#})^{\#} = L$ is given in \cite[Lemma 1.5(ii), p203]{Scharlau_book} and also \cite[\S82F, pp230-231]{OMeara:1971zr}.  From this and Lemma \ref{lem:lattice_scaling_and_duality} we see that $(L^{\#\a})^{\#\a} = \a (\a L^{\#\a})^\# = {\a \cdot \a^{-1}} (L^\#)^\# = L$, proving the lemma.
\end{proof}

\smallskip
\begin{defn} \label{def:a_modular}
We say that a bilinear lattice $L \subset (V,B)$ is {\bf $\a$-modular} for some fractional ideal $\a$ if $L^\# = \frac{1}{\a} L$.  This is equivalent to saying that the $\a$-dual lattice $L^{\#\a} = L$.
\end{defn}

\begin{lem} [Modular lattice value ideals]\label{Lem:modular_bilinear_value_ideal}
If $L \subset (V,B)$ is an $\a$-modular bilinear lattice then its bilinear value ideal $B(L,L) = \a$.
\end{lem}

\begin{proof}
This follows from \cite[\S82:14, pp232]{OMeara:1971zr} since their definition of $\a$-modular on page 231 includes that the scale ideal $\mathfrak{s}(L) := B(L,L) = \a$.
\end{proof}

\begin{lem}[Scaling modular lattices] \label{Lem:lattice_scaling_and_modular_lattices}
Suppose that $\a$ and $\b$ are non-zero fractional ideals of $F$.  
If $L  \subset (V,B)$ is an $\a$-modular bilinear lattice then the scaled lattice $\b L  \subset (V,B)$ is a $(\b^2 \a)$-modular lattice.
\end{lem}

\begin{proof}
By Definition \ref{def:a_modular} it is enough to show that $(\b L)^{\#\b^2 \a} = \b L$, but from 
Lemma \ref{lem:lattice_scaling_and_duality} we know that 
$(\b L)^{\#\b^2 \a} = \b^2(\b L)^{\#\a} = \b^2 L^{\#\a\b^{-1}} = \b L^{\#\a} = \b L$
since %$L^{\#\a} = L$ because 
$L$ is $\a$-modular.
\end{proof}

\begin{rem}[Relation with O'Meara's notation]
At the request of the referee, we include some comments about how our notation relates to the notions found in O'Meara's book \cite{OMeara:1971zr}.
Our notion of an $\a$-modular lattice coincides with O'Meara's by \cite[\textsection 82:14, p232]{OMeara:1971zr}, and our bilinear value ideal $B(L,L)$ is O'Meara's scale ideal $\s(L)$ \cite[\textsection 82E, p227]{OMeara:1971zr}.  However our notion of ``scaling'' differs from the notion in \cite[\S82J, p238]{OMeara:1971zr} since given a bilinear lattice $L \subset (V,B)$
%
 %there 
 O'Meara denotes by $L^\al$ the lattice $L$ in an ambient bilinear space $(V, \al B)$ whose values are scaled by $\al \in F^\times$, whereas our notion  of scaling fixes the ambient bilinear space $(V,B)$ and scales the lattice $L$ within it.
\end{rem}

\smallskip
\begin{lem}[Polarization Identity]
Given a quadratic form $Q(\x)$ in $n$ variables over a ring $R$, we can associate to it the {\bf Hessian (symmetric) bilinear form} $H(\x,\y)$ defined by the polarization identity
$$
Q(\x + \y) = Q(\x) + H(\x,\y) + Q(\y),
$$
which satisfies $H(\x,\x) = 2Q(\x)$ and also $H(\x, \y) = (\x)^t A \y$ where the matrix $A \in \mathrm{Sym}^n( R )$ is defined by 
$A := (a_{ij})$ with $a_{ij} := \frac{\partial^2 A}{\partial x_i \partial x_j}$.
\end{lem}

\begin{proof}
This follows easily when $2$ is invertible in $R$ since we can write $Q(\x) = {\x}^t A \x$ for some $A \in \mathrm{Sym}_n(R )$, and can be verified in general by writing $Q(\x) = \sum_{i \leq j} c_{ij} x_i x_j$ with $c_{ij} \in R$ and evaluating $Q(\x + \y) - Q(\x) - Q(\y)$.
\end{proof}

%%%%%%%%%%%%%%%%
%%  Discriminant Modules   %%
%%%%%%%%%%%%%%%%

\subsection{Discriminant modules}
We now describe a finite module associated to an $\a$-valued bilinear lattice $L$ whose geometry will be very useful later for constructing maximal $\a$-valued superlattices of $L$.

%% Defining the discriminant bilinear module
\begin{defn}[Discriminant module]
Suppose that $F$ is a number field, $\a$ is a non-zero fractional ideal of $F$ and $L \subset (V,B)$ is an $\a$-valued bilinear $\O_F$-lattice.  Then we define the {\bf $\a$-discriminant module of $L$} as the bilinear module $\D_\a := \D_\a(L) := L^{\#\a}/L$ of $L$ equipped with the $(F/\a)$-valued bilinear form $\widetilde{B}(\x + L, \y + L) := B(\x, \y) + \a$ induced from $B$ on $V$.
\end{defn}

\begin{lem}[Non-degeneracy]
If $L$ is a non-degenerate $\a$-valued bilinear lattice, then its (bilinear) $\a$-discriminant module $L^{\#\a}/L$ is also non-degenerate.
\end{lem}

\begin{proof}
Suppose $\v\in L^{\#\a}$ and $B(\v, L^{\#\a}) \in \O_F$.  Then $\v \in (L^{\#a})^{\#\a} = L$ by Lemma \ref{Lem:double_dual}, so $\v = \vec 0 \in \D_\a(L)$.
\end{proof}

%% Anisotropy Lemma
\begin{lem} \label{Lem:a_discriminant_isotropy_correspondence}
Suppose that $L$ is an $\a$-valued lattice in a non-degenerate bilinear space $(V,B)$ over a number field $F$.
Then there is a bijective inclusion-preserving correspondence %between
$$
\left\{
\begin{matrix} \text{$\a$-valued lattices $L'$} \\ \text{with $L \subsetneq L' \subseteq (V,B)$} \end{matrix} 
\right\}
\quad
\longleftrightarrow 
\quad
\left\{
\begin{matrix}\text{isotropic submodules} \\  \text{$L'/L \subseteq (L^{\#\a}/L, \widetilde{B})$}\end{matrix}
\right\}.
$$ 
\end{lem}

\begin{proof}
%Certainly we see that t
There is an inclusion-preserving bijection between lattices $L'$ with $L \subsetneq L' \subseteq L^{\#\a}$ and non-zero submodules $L'/L$ of $L^{\#\a}/L$, and the $\a$-valued 
extra 
condition follows because %of the equivalence 
$
\text{$L'/L$ is isotropic for $\widetilde{B}$}  \Longleftrightarrow  B(L') \in \a  %\Longleftrightarrow  Q(L') \in \tfrac{1}{2}\Z.
$.
\end{proof}

%% Sublatice Lemma
\begin{lem}[Sublattice discriminants] \label{Lem:sublattice_discriminants}
If $L$ is an $\a$-valued lattice in a non-degenerate bilinear space over a number field and $L'$ is a (finite index) sublattice of $L$, then 
$$
|\D_{\a}(L')| = [L:L']^2 \cdot |\D_{\a}(L)|.
$$
\end{lem}

\begin{proof}
Since $L$ is $\a$-valued, we have the inclusions
$
L' \subseteq L \subseteq L^{\#\a} \subseteq (L')^{\#\a},
$
and by the non-degeneracy of $B$ we have that 
$
[L':L] = [(L')^{\#\a}:L^{\#\a}].
$
%(e.g. see \cite{}).
Therefore
$$
|\D_{\a}(L')| = [(L')^{\#\a}:L'] = [L':L]^2 \cdot [L^{\#\a}:L] = [L':L]^2 \cdot |\D_{\a}(L)|.
$$
\end{proof}

%%%%%%%%%%%%%%%%%%%%%%%
%% SECTION:  Maximal Bilinear Lattices  %%
%%%%%%%%%%%%%%%%%%%%%%%
%%%%%%%%%%%%%%%%%%%%%%%
%% SECTION:  Maximal Bilinear Lattices  %%
%%%%%%%%%%%%%%%%%%%%%%%
\section{Maximal Bilinear Lattices}

In this section we describe how to produce an $\a$-maximal lattice in any given non-degenerate bilinear space $(V, B)$ over a number field $F$.  
By Lemma \ref{Lem:a_discriminant_isotropy_correspondence} we can 
do this by finding a maximal isotropic submodule of 
its $\a$-discriminant module 
$\D_\a$.

%\subsection{Bilinear Spaces -- Simplifying discriminant modules using scaled dual lattices}
\subsection{Finding Saturated discriminant modules}

We begin by performing a series of lattice operations to arrange that the $\a$-discriminant module is a product of bilinear spaces over finite fields $\F_q$.
This is a common first step in many 
algorithms to find %maximal lattices or 
maximal orders (e.g. \cite[\textsection6.1]{CohenBook:1993}, \cite[\textsection2.4.1]{CohenBook:2000} \cite[\textsection7]{Voight:2010uq}). 
While our approach uses duality on global lattices, we can more clearly see the meaning of saturation in terms of local lattices, where it is just given by scaling certain summands of the local Jordan decomposition.

\begin{defn}
Suppose $F$ is a number field and $\a$ is a non-zero fractional ideal of $F$.
We say that an $\a$-valued bilinear $\O_F$-lattice $L$ is {\bf $\a$-saturated} if for every non-zero prime ideal $\p$ of $\O_F$  the 
{\bf local $\a$-discriminant module}
$(\D_\a)_\p := \D_\a(L) \otimes_{\O_F} \O_\p$ is annihilated by $\p$.
\end{defn}

\subsubsection{Local perspective on $\a$-saturated lattices}

For perspective we include the following algorithm to construct an $\a$-saturated local lattice containing a given $\a$-valued local lattice.  The main idea is 
that scaling a bilinear lattice $L$ by an ideal $\b$ alters its values by the square of that ideal
(i.e. $B(\b L,\b L) = \b^2 B(L,L)$), so by a series of local scalings we can adjust $L$ so that its modular components are at most one valuation larger than the valuation of the desired value ideal $\a$.

\begin{alg}[Finding a Local $\a$-saturated lattice] \label{Alg:find_local_a_saturated_lattice}
Given an $\a$-valued bilinear $\O_\p$-lattice $L_\p$ in a non-degenerate bilinear space $(V, B)$ over $F_\p$, we give an algorithm for finding an $\a$-saturated superlattice of $L_\p$.
\begin{enumerate}
\item Compute a Jordan decomposition $L_\p = \oplus_{i \in \Z\geq 0} L_{i, \p}$ where the $L_{i, \p}$ are $\a\p^i$-modular  %By assumption we know that 
\\ (e.g. using \cite[(2.2) and Lemma 2.1, pp354-5]{Hanke:2004br} or \cite[\textsection 94, p279-280]{OMeara:1971zr}).
\item Return $L'_\p = \oplus_{i \in \Z\geq 0} \p^{-\nu_i}L_{i,\p}$ where $\nu_i := \lfloor \tfrac{i}{2}\rfloor$.
\end{enumerate}
\end{alg}

\begin{proof}
We know that $i\geq 0$ for all non-zero $L_{i,\p}$ since for these we have $B(L_{i,\p}, L_{i,\p}) = \a\p^i \subseteq \a$ by Lemma \ref{Lem:modular_bilinear_value_ideal}.  
%(FINISH PROOF!)
\end{proof}

\begin{rem} [Local-Global approach for $\a$-saturated lattice]\label{Rem:local_global_for_a_saturated_lattices}
One could use Algorithm \ref{Alg:find_local_a_saturated_lattice} to construct a global $\a$-saturated bilinear lattice by first choosing some $\a$-valued lattice $L$, computing the finite set of primes $\SS$ where $L_\p$ is not $\a$-saturated, using Algorithm \ref{Alg:find_local_a_saturated_lattice} to construct $\a$-saturated local lattices $L'_\p \supseteq L_\p$ for all $\p \in \SS$, and then using an algorithm (e.g. Algorithm \ref{alg:Local_global_algorithm_for_lattices}) to construct the unique $\O_F$-lattice $L''$ satisfying
$$
(L'')_\p = 
\begin{cases}
L'_\p & \text{if $\p\in \SS$,} \\
L_\p & \text{if $\p\notin \SS$.} 
\end{cases}
$$
\end{rem}

\bigskip
%\medskip
For the interested reader, we include a reference to an algorithm for performing the local-global construction with lattices, which is essentially an algorithm for Strong Approximation on $\GL_n$ over a number field.

\begin{alg}[Local-Global Algorithm for lattices] \label{alg:Local_global_algorithm_for_lattices}
Suppose that $V$ is a finite-dimensional vector space over a number field $F$.
Given (local) $\O_\p$-lattices $L'_\p \subseteq V_\p$ for all non-zero prime ideals $\p$ of $\O_F$ satisfying the compatibility condition that for all but finitely many $\p$ we have that $L'_\p = L''$ for some fixed $\O_F$-lattice $L'\subset V$, there is an algorithm for constructing a (global) $\O_F$-lattice $L\subseteq V$ so that its local lattices 
$L_\p := L \otimes_{\O_F} \O_\p$  satisfy $L_\p = L'_\p$ (as subsets of $V_\p$) for all $\p$.
\end{alg}

\begin{proof}
A constructive algorithm for this is given as the proof of \cite[\S84:14, pp218-219]{OMeara:1971zr}, which uses the algorithm in 
 \cite[\S81:11, pp214-216]{OMeara:1971zr} to normalize the presentation of two lattices in a common global vector space.
\end{proof}

\subsubsection{Global algorithms for finding an $\a$-saturated lattice}

In this section we give a global algorithm using a scaled dual lattice construction  
to simultaneously reduce the divisibility of the annihilator ideal of the $\a$-discriminant module $\D_\a$ at all primes $\p$.
By repeating this process until the annihilator of $\D_\a$ is a squarefree ideal, we obtain the desired $\a$-saturated lattice.

\begin{lem} \label{Lem:integral_scaled_dual}
If $M$ is an $(\a\p^i)$-modular bilinear lattice in a non-degenerate bilinear space $(V, B)$ over a number field $F$, then for any $\lambda \in \Z$ the scaled dual lattice $M' := \a\p^\lambda M^{\#}$ is  $(\a\p^{2\lambda - i})$-modular, and $M'$ is $\a$-valued (as a bilinear lattice) iff $\lambda \geq \lceil\tfrac{i}{2}\rceil$.
\end{lem}

\begin{proof}
Since $M$ is an $(\a\p^i)$-modular lattice %by Remark \ref{Rem:a_dual_from_dual} we know that 
we have that $M^\# = \tfrac{1}{\a\p^{i}} M$ and so $M' := \p^{\lambda-i} M$.  By Lemma \ref{Lem:lattice_scaling_and_modular_lattices}, we have that $M'$ is $(\p^{2(\lambda - i)}\a\p^i)=(\a\p^{2\lambda - i})$-modular.  
This shows that the ideal 
$B(M', M') =  B(\a\p^{2\lambda - i}(M')^\#, M') = \a\p^{2\lambda - i}$
is contained in $\a$ iff $\lambda \geq \lceil\tfrac{i}{2}\rceil$.
\end{proof}

From Lemma \ref{Lem:integral_scaled_dual} we see that the scaled dual lattice $(\p^\al L_i^{\#\a}, B)$ is proper $\a$-valued superlattice of $L_i$ when $i> \al \geq \lceil\tfrac{i}{2}\rceil$, so in particular we can use this idea on $\a\p^i$-modular lattices only when $i \geq \lceil\tfrac{i }{2}\rceil$, which only happens when $i \geq 2$.
To make this idea useful for more general $\a$-valued bilinear lattices, we make the following definition.

\begin{defn}[Maximal scale index]
Given a non-zero fractional ideal $\a$ in a number field $F$, and a non-degenerate $\a\cdot \O_\p$-valued bilinear $\O_\p$-lattice $L$,
we define the {\bf maximal scale index} $m_{\p, \a}$ of $L$ at $\p$ (relative to $\a$) to be the largest integer $i$ so that the
$\a\p^{i}$-modular component of $L$ is non-zero 
(in any Jordan decomposition $L \cong \oplus_{i\in\Z} L_i$ over $\O_\p$ where the $L_i$ are $\p^i$-modular).
Since $L$ is $\a$-valued, we always have $m_{\p, \a} \geq 0$.
\end{defn}

We now explain how to construct an $\a$-saturated lattice from any $\a$-valued 
lattice.

\begin{thm} \label{Thm:a_dual_superlattice_thrm}
Suppose that 
$L$ is an $\a$-valued lattice 
in a non-degenerate bilinear space over a number field $F$, 
whose maximal scale index $m_{\p,\a} \geq 2$ for some non-zero prime ideal $\p$ of $\O_F$.  Then the bilinear lattice 
$$L' := L + \p^\lambda L^{\#\a} \qquad\text{ with }\quad  \lambda := \lceil \tfrac{m_{\p,\a}}{2}\rceil$$
 is a proper $\a$-valued superlattice of $L$, for which the maximal scale index $m'_{\p, \a} = \lambda  < m_{\p, \a}$.  
 Therefore we can construct an $\a$-valued superlattice $L''$ of $L$ with maximal scale index $m''_{\p, \a} \leq 1$ for all primes $\p$.
\end{thm}

\begin{proof}
Let $\al = \ord_\p(\a)$ and let $L \cong \oplus_{i\in\Z} L_i$ be a Jordan decomposition of $L$ over $\O_\p$ where the $L_i$ are $\a\p^i$-modular lattices.  By assumption we know that 
$L_i = \{0\}$ unless $0 \leq i \leq m_{\p, \a}$.  Also since duality and scaling both preserve the Jordan decompositon we can compute the lattice $L'$ in each Jordan summand.
Suppose that $M$ is a $(\p^{\al + i})$-modular Jordan summand of $L$.  Then by Lemma \ref{Lem:integral_scaled_dual} the associated scaled dual lattice $\p^\lambda M^{\#\a}=\a\p^\lambda M^{\#}$ is $(\p^{2\lambda + \al - i}) = (\p^{\al + i + 2(\lambda - i)})$-modular, so the sum $M' := M + \a\p^\lambda M^{\#}$ is $\p^{\al + i + \min\{0, 2(\lambda - i)\}}$-modular.

We now compute the maximal scale index $m'_{\p, \a}$ for $L'$ by varying over all possible Jordan summands $M$.  When $i \leq \lambda$ we have $M' = M$ which is $\p^{\al + i}$-modular with largest power $\al + \lambda$, but when $i > \lambda$ we have that $M'$ is $\p^{\al + 2\lambda - i}$-modular with the largest power here being 
$\al + \lambda - 1$, giving  $m'_{\p, \a} = \lambda$.  By repeatedly applying this procedure, we obtain a decreasing sequence of $m'_{\p, \a}$ until $m'_{\p, \a}\leq 1$ (since $\lceil\frac{x}{2}\rceil < x$ when $x \in \Z \geq 2$).
\end{proof}

\begin{alg}[Construct an $\a$-saturated lattice] \label{Alg:find_saturated_a_valued_lattice}
Given a non-degenerate bilinear space $(V,B)$ over a number field $F$ and a non-zero fractional ideal $\a$ of $F$, we give an algorithm for finding a saturated $\a$-valued bilinear $\O_F$-lattice $L$ in $(V,B)$.
\begin{enumerate}
\item  Choose an arbitrary $\O_F$-lattice $L \subset (V, B)$ by repeatedly choosing vectors $\v_k$ not in the $F$-span of the previous vectors $\{\v_1, \dots, \v_{k-1}\}$ until their $F$-span is $V$.  Then let $L := \Span_{\O_F}\{\v_1, \dots, \v_{n}\}$, where $n = \dim_F(V)$.
\item Let $\b :=B(L,L)$.  If $\b \nsubseteq \a$ then replace $L$ by $\c L$ where $\c$ is any non-zero fractional ideal satisfying $\c^2\b \subseteq \a$. 
(For convenience we often we take $\c$ to be principal.) 
\item \label{Item:Saturation} 
If $L$ is saturated then return $L$.  Otherwise, let $\l$ be the integral ideal $\l:= \prod_{\p} \p^{\lambda_\p}$ where $\lambda_\p := \lceil \tfrac{m_{\p, \a}}{2} \rceil$ and let $L' := L + \l L^{\#\a}$.
\item Replace $L$ by $L'$ and repeat step \ref{Item:Saturation}.  
\end{enumerate}
This algorithm terminates by Theorem \ref{Thm:a_dual_superlattice_thrm}, since there are only finitely many primes $\p$ where $m_{\p, \a} \neq 0$.
\end{alg}

\subsection{Finding a Maximal Isotropic subspace}

To pass from an $\a$-saturated bilinear lattice to an $\a$-maximal lattice, we must be able to compute a maximal isotropic subspace of a bilinear space over a finite field.  In this section we describe how to do this, with special attention to when the characteristic of the base field is two.

\begin{rem}[Bilinear and semilinear quadratic forms in characteristic two]  \label{rem:bilinear_and_quadratic_in_char_2}
Given a non-degenerate bilinear space $(V,B)$ over a field $K$ of characteristic two, the polarization identity ensures that the associated  quadratic form $Q_B(\x) := B(\x, \x)$ is a homomorphism of additive groups (since $2B(\x+\y, \x+y) = 0$).  In general $Q_B$ is not a $K$-linear map since $Q_B(\al \x) = \al^2 Q_B(\x) \neq \al Q_B(\x)$.  However when $K$ is a perfect field (e.g. $K$ is a finite field) we know that squaring 
is a field automorphism $\sigma$ of $K$ (called the Frobenius automorphism), so the quadratic scaling property can be rewritten as 
$Q_B(\al\x) = \sigma(\al) \cdot Q_B(\x)$, which says that the map $Q_B$ is a ``semilinear" map (see \cite[\S4.1, p28]{OMeara_Lectures_on_linear_groups}).  Therefore
$Q_B$ ``behaves like'' a $K$-linear map in that the image and pre-image of any $K$-subspace are again $K$-vectorspaces.  We also have that $\dim_K(\Im(Q_B)(V)) + \dim_K(\Ker(Q_B)(V)) = \dim_K(V)$.  This observation will be extremely useful for understanding the concept of isotropy in these bilinear spaces.
\end{rem}

\medskip
We begin by recalling a well-known test (at least in odd characteristic) for the existence of a (non-zero) isotropic vector in a non-degenerate  quadratic/bilinear space over a finite field.

\begin{lem}[Isotropy testing over $\F_q$] \label{lem:Check_isotropy_over_F_q}
Given a non-degenerate bilinear space (V,B) over a finite field $\F_q$, we can use the following criteria to determine 
if $(V,B)$ is isotropic (i.e. has a non-zero vector $\v\in V$ with $B(\v, \v) = 0$):
\begin{enumerate}
\item If $\dim(V) \geq 3$ then $(V,B)$ is isotropic.
\item If $\dim(V) = 2$ 
then 
$$\text{$(V,B)$ is isotropic $\iff \det(B) = -1(\F_q^\times)^2$.}$$
When  $\Char(\F_q) = 2$ then $(V,B)$ is isotropic.
\item If $\dim(V) = 1$ then $(V,B)$ is not isotropic.
\end{enumerate}
\end{lem}
\begin{proof} First suppose that  $\Char(\F_q) \neq 2$.  Then when $n \geq 3$ this is \cite[62:1b, p158]{OMeara:1971zr}.  When $n=2$ then by \cite[\S62:1, p157]{OMeara:1971zr} there are exactly two non-degenerate quadratic forms over $\F_q$, the anisotropic norm form $N_{\F_q^2/\F_q}(\x)$ and the (isotropic) hyperbolic plane.

Now suppose that $\Char(\F_q) = 2$.  By Remark \ref{rem:bilinear_and_quadratic_in_char_2} we see that the kernel of the map $Q_B(\x) := B(\x, \x)$
is an $\F_q$-subspace of dimension $\geq \dim(V) - 1$.  Therefore when $\dim(V) \geq 2$ we have that $K \neq \{\vec 0\}$ and so $(V, B)$ is isotropic.  

Finally, when $n=1$ then $B(x, x) = \al x^2$ with $\al \neq 0$ iff $(V, B)$ is isotropic.  
\end{proof}

Once we determine that a bilinear space over $\F_q$ is isotropic,
we need a way to find isotropic vectors in it.   Over a finite field there are only finitely many lines, so an exhaustive search through all such lines is possible.  We give a somewhat better algorithm that essentially checks if a random (rational) line intersects the projective quadric hypersurface $Q_B(\x) =0$.  Here the probability of success (for each attempt) is roughly $50\%$ since by the quadratic formula the existence of an intersection is governed by whether the discriminant of the associated quadratic polynomial is a square in $\F_q$.

\begin{alg}[Finding an isotropic vector over $\F_q$] \label{alg:find_isotropic_vector}
Given an isotropic non-degenerate bilinear space $(V,B)$ over $\F_q$, we give a randomized algorithm for finding some non-zero $\v \in V$ so that $Q_B(\v) :=B(\v,\v)= 0$.
\begin{enumerate}
\item  Randomly choose two linearly independent vectors $\vec a, \vec m \in V$.
\item  Compute the intersection of the line $L := \{\vec a + t\vec m \mid t \in \F_q\} \subseteq V$ with $Q_B(\x)=0$ by solving $Q_B(\vec a + t\vec m) = 0$ for $t$ over $\F_q$.
\item If a solution $t_0 \in \F_q$ exists, then the vector $\v := \vec a + t_0 \vec m$ is isotropic.  Otherwise repeat from step 1.
\end{enumerate}
\end{alg}

\begin{proof}
Since $\vec a$ and $\vec m$ are linearly independent, we know that the affine line $L$ descends to a line in the projective space $\P(V)$, and that any vector $\v \in L$ is non-zero.
\end{proof}

We can now easily find a maximal isotropic subspace in any bilinear space over a finite field of characteristic $\neq 2$, since there any (non-zero) isotropic vector is contained within a hyperbolic plane.

\begin{alg}[Isotropic vector $\Rightarrow$ Hyperbolic plane] \label{alg:isotropic_gives_hyperbolic_plane}
Given a non-degenerate bilinear space $(V, B)$ and a (non-zero) isotropic vector $\v \in V$,  we give an algorithm to find some $\w \in V$ so that the subspace of $V$ spanned by the basis $\B = \{\v, \w\}$ is a hyperbolic plane.
\begin{enumerate}
\item Choose a basis $\B$ for $V$ whose first basis vector is $\v$.
\item Find some $\w \in \B$ so that $B(\v, \w) = 0$.
\item Scale $\w$ so that $B(\v,\w) =1$  (i.e. replace $\w$ by $\frac{\w}{B(\v,\w)}$).
\item Shear $\w$ by $\v$ to arrange that $B(\w, \w) = 0$ (i.e. replace $\w$ by $\w - B(\v,\w)\cdot \v$).
\end{enumerate}
\end{alg}

\begin{proof}
Since $\v \neq \vec0$ we can extend $\{\v\}$ to a basis $\B = \{\v_1, \dots, \v_n\}$ with $\v_1 = \v$ in Step 1.  If $B(\v, \v_i) = 0$ for all $i \geq 2$ then by linearity we have $B(\v, V) = \{0\}$ and so $\v \in \Rad(V) = \{\vec 0\}$, which cannot happen since $\v \neq \vec0$.  Therefore some $B(\v, \v_i) =0$ in Step 2.  The remaining steps follow by linearity.
\end{proof}

\begin{alg}[Maximal Totally Isotropic subspaces in odd characteristic] \label{alg:maximal_totally_iso_in_char_not_2}
Suppose that $(V,B)$ is a 
%non-degenerate 
bilinear space over a finite field of characteristic $\neq 2$.  Then we give an algorithm to find an orthogonal decomposition of $V$ as $V = R \perp H \perp A$ where $R := \Rad(V)$, $H$ is a hyperbolic space, and $A$ is anisotropic.
\begin{enumerate}
\item Compute the radical subspace $R$ of $V$, and find a complementary subspace $V'$ of $R$ in $V$.  Then $(V', B\mid_{V'})$ is a non-degenerate bilinear space.
\item Use Lemma \ref{lem:Check_isotropy_over_F_q} and Algorithms \ref{alg:find_isotropic_vector} and \ref{alg:isotropic_gives_hyperbolic_plane} to repeatedly split off hyperbolic planes from $V'$ until the remaining bilinear space is anisotropic (and take this as $A$).
\end{enumerate}
Further, by writing the hyperbolic space $H$ as a (non-orthogonal) direct sum of two Lagrangian subspaces $H = L \oplus L'$ (which is done implicitly in Step 2), we have that the subspace $M := R + L$ is a maximal totally isotropic subspace of $(V,B)$.
\end{alg}

\begin{proof}
Certainly $M$ is a totally isotropic subspace since for any $\vr \in R$ and $\vl \in L$ we have 
$B(\vr + \vl, \vr + \vl) = B(\vl, \vl) = 0$.  To see maximality, suppose we have some totally isotropic subspace $M'$ of $V$ with $M'\supseteq M$ and write any $\v' \in M'$ as $\v = \vr + \vl + \vl' + \va$ according to the decomposition above.  Then we must have $\vl'=\vec 0$ since otherwise we can find some $\vl_0 \in L$ so that $B(\v, \vl_0) = B(\vl, \vl_0) \neq 0$.  Similarly $B(\v, \v) = B(\va,\va) = 0 \implies \va = \vec 0$ since $A$ is anisotropic.  This shows that $\v \in M$, so $M' \subseteq M$ and $M$ is a maximal totally isotropic subspace of $V$.
\end{proof}

Over finite fields of characteristic two several new phenomena arise, including the failure of the Algorithm \ref{alg:isotropic_gives_hyperbolic_plane} due to the presence of ``metabolic" bilinear spaces (e.g. $B(\x, \y) = x_1y_2 + x_2 y_1 + x_2y_2$) that are isotropic but not hyperbolic.  However to compensate for this complication, we gain the simplification that quadratic forms behave like linear forms in this setting (see Remark \ref{rem:bilinear_and_quadratic_in_char_2}).  With this in mind, we look for a maximal totally isotropic subspace of any bilinear space over a finite field of characteristic two.

\begin{alg}[Maximal Totally Isotropic subspaces in characteristic two] \label{alg:maximal_totally_iso_in_char_2}
Suppose that $(V,B)$ is a 
%non-degenerate 
bilinear space over a finite field $K$ of characteristic two.  Then we give an algorithm to find an orthogonal decomposition of $V$ as $V = R \perp (I \oplus A)$ where $R := \Rad(V)$, $I$ is a totally isotropic space, and $A$ is anisotropic with $\dim(A)\leq 1$.  
For convenience we define $Q_B(\x) := B(\x,\x)$.
\begin{enumerate}
\item Compute the radical subspace $R$ of $V$, and find a complementary subspace $V'$ of $R$ in $V$.  Then $(V', B\mid_{V'})$ is a non-degenerate bilinear space.
\item Choose a basis $\B = \{\v_1, \dots, \v_n\}$ for  $V'$, ordered so that $Q_B(\v_n) \neq 0$ if some $Q_B(\v_i) \neq 0$.
\item If $Q_B(\v_n) \neq 0$ then by shearing the $\v_i$ by $\v_n$ we can arrange that $Q_B(\v_i) = 0$ for all $1 \leq i < n$  (i.e. replace $\v_i$ by $\v_i - \sqrt{\frac{Q_B(\v_i)}{Q_B(\v_n)}} \v_n$).
\item If $Q_B(\v_n) = 0$ then set $I:=\Span_K(\{\v_1, \dots, \v_n\})$ and $A:= \{\vec 0\}$,  
otherwise set $I:=\Span_K(\{\v_1, \dots, \v_{n-1}\})$ and $A:= \Span_K(\{\v_n\})$.
\end{enumerate}
Given this decomposition, we have that the
subspace $M := R + I$ is a maximal isotropic subspace of $(V,B)$. 
\end{alg}

\begin{proof}
We know that $M$ is a totally isotropic subspace since for any $\vr \in R$ and $\vi \in I$ we have 
$B(\vr + \vi, \vr + \vi) = B(\vi, \vi) = 0$.  To see maximality, suppose we have some totally isotropic subspace $M'$ of $V$ with $M'\supseteq M$ and write any $\v' \in M'$ as $\v = \vr + \vi + \va$ according to the decomposition above.  Then $0 = Q_B(\vr + \vi + \va) = Q_B(\va)$ gives $\va = \vec 0$ and so $\v\in M$.  Therefore $M' \subseteq M$ and $M$ is a maximal totally isotropic subspace of $V$.
\end{proof}

\begin{rem}
It is interesting to note that the characteristic two Algorithm \ref{alg:maximal_totally_iso_in_char_2} is both stronger and weaker than its odd characteristic counterpart (Algorithm \ref{alg:maximal_totally_iso_in_char_not_2}).  It is stronger in that we more easily see the totally isotropic subspace and $\dim(A) \leq 1$ (instead of $\leq2$), but it is weaker in that we have concluded much less about the underlying bilinear space under the decomposition.  This illustrates the general phenomenon that quadratic forms (i.e. questions about isotropy) and symmetric bilinear forms (i.e. questions about orthogonality) are distinct notions in characteristic two.
\end{rem}

\subsection{Finding a Maximal bilinear lattice}

Now that we can find a maximal isotropic subspace of $\a$-discriminant module of a saturated $\a$-valued bilinear lattice $L$, we are ready to compute a maximal $\a$-valued bilinear superlattice of $L$.

\begin{alg}[Construct a maximal $\a$-valued bilinear lattice]  \label{alg:construct_maximal_bilinear_lattice}
Given a non-degenerate bilinear space $(V,B)$ over a number field $F$ and a non-zero fractional ideal $\a$ of $F$, we give an algorithm for finding a maximal $\a$-valued bilinear $\O_F$-lattice $L$ in $(V,B)$.
\begin{enumerate}
\item Use Algorithm \ref{Alg:find_saturated_a_valued_lattice} 
(or Algorithm \ref{Alg:find_local_a_saturated_lattice} and Remark \ref{Rem:local_global_for_a_saturated_lattices}) 
to find an $\a$-saturated lattice in $(V,B)$.
Let $\SS$ denote the finite set of (non-zero) primes $\p$ of $\O_F$ where the local discriminant module $(\D_\a(L))_\p \neq \{0\}$.  
\item For each $\p\in \SS$, use Algorithm \ref{alg:maximal_totally_iso_in_char_not_2} or \ref{alg:maximal_totally_iso_in_char_2} to find a maximal isotropic subspace $I_\p$ of the non-zero $\F_\p$-vector space $(\D_\a(L))_\p$. %$(L^{\#\a}/L) \otimes_{\O_F} \O_\p$.
%(\frac{1}{\p}\O_\p/\O_\p)$.
\item For each $\p\in \SS$, choose a lift of a basis $\B_\p$ for $I_\p$ to a set $\B'_\p \subseteq L^{\#\a}$ of vectors so that for every $\q\in\SS$
their $\F_\q$-span in $(\D_\a(L))_\q$ satisfies 
$$
\Span_{\F_\q}(\B'_\p) = 
\begin{cases}
I_\p & \text{if $\q= \p$,} \\
\{0\} & \text{if $\q\neq \p$.}
\end{cases}
$$
This can be done with a constructive version of the Chinese Remainder Theorem. 
% (e.g. \cite[]{CRT}).
\item Then the lattice $L' := L + \Span_{\O_F} (\cup_{\p\in \SS} \B'_\p)$ is a maximal $\a$-valued bilinear $\O_F$-lattice containing $L$.
\end{enumerate}
\end{alg}

\begin{proof}
Since $L^{\#\a} \supseteq L' \supseteq L$ and $L'/L \subseteq (\D_\a)_\p$ is a maximal isotropic submodule, the correspondence in Lemma \ref{Lem:a_discriminant_isotropy_correspondence} ensures that $L'$ is a maximal $\a$-valued lattice in $(V,B)$.  
\end{proof}

We conclude with a very general lemma about the uniqueness of discriminant modules of maximal $\a$-valued bilinear lattices which will be useful when thinking about maximal quadratic lattices.  This proof
was explained to the author by Prof. Gabrielle Nebe.

\begin{lem}[Uniqueness of Maximal Lattice Discriminant modules]  \label{lem:uniqueness_of_maximal_discriminant_modules}
Suppose that $L_1$ and $L_2$ are maximal $\a$-valued lattices in a non-degenerate bilinear space $(V,B)$ over a field $F$.  Then the discriminant modules $\D_\a(L_1) \cong \D_\a(L_2)$ as $(F/\a)$-valued bilinear $\O_F$-modules.
\end{lem}

\begin{proof}
We first notice that the discriminant modules $\D_\a(L_1) \cong \D_\a(L_1\cap L_2) \cong \D_\a(L_2)$ in the Witt group of bilinear torsion $\O_F$-modules (i.e. up to weakly metabolic summands) since the submodules $N_i := L_i / (L_1 \cap L_2)$ of $\D_\a(L_1\cap L_2)$ are isotropic with $(N_i)^\perp = L_i^{\#\a} / (L_1 \cap L_2)$, so by \cite[Lemma 1.4]{Scharlau_book} we know that the orthogonal sums $\D_\a(L_i) \perp -\D_\a(L_1\cap L_2)$ are weakly metabolic.

For convenience, let $\D_i := \D_\a(L_i)$.  
Since the sum 
$\D_1 \perp -\D_2$
is weakly metabolic, we know that it has some submodule $N$ with $N = N^\perp$ and further that $|\D_1|\cdot|\D_2| = |N| \cdot|N^\perp| = |N|^2$
by comparing cardinalities in the exact sequence of torsion $\O_F$-modules 
$$
0 \longrightarrow N^\perp \longrightarrow (\D_1 \perp -\D_2) \xrightarrow{B(\cdot, (x_1, -x_2))} N^* := \Hom_{\O_F}(N, F/\O_F) \longrightarrow 0.
$$
We also know that the projections $\pi_i:N\rightarrow \D_i$ are injective since the kernel lies in $\D_{i'} \cap N = \{0\}$ where $i \neq i'$, so by the formula $|\D_1|\cdot|\D_2| = |N|^2$ we see that the $\pi_i$ are also surjective.  To see that $\pi_2 \circ \pi_1^{-1}: \D_1 \rightarrow -\D_2$ is an isomorphism of bilinear modules, suppose that it maps $x \mapsto y$ and $x' \mapsto y'$.  Then since $x+y, x' + y' \in N$, we have that $0 = B(x+x', y+y') = B(x,x') + B(y,y')$, giving $B(x,x') = -B(y,y')$, which proves the lemma.
\end{proof}

%%%%%%%%%%%%%%%%%%%%%%%%
%%  SECTION:  Maximal quadratic lattices    %%
%%%%%%%%%%%%%%%%%%%%%%%%
%\newpage
%%%%%%%%%%%%%%%%%%%%%%%%
%%  SECTION:  Maximal quadratic lattices    %%
%%%%%%%%%%%%%%%%%%%%%%%%
\section{Maximal Quadratic Lattices}

One application of Algorithm \ref{alg:construct_maximal_bilinear_lattice}
is to help with finding a maximal $\a$-valued quadratic lattice (in a 
quadratic space) over a number field $F$ where the prime $p=2$ is unramified.
To do this we require one additional notion for bilinear lattices.

\begin{defn}[Even bilinear lattices]
Given a lattice $L$ in a bilinear space $(V,H)$, we 
respectively define the {\bf full} and {\bf partial $\a$-even subsets} of $L$ as
$$
L_{\a\even} := \{\x \in L \mid H(\x, \x) \in 2\a \},   %\text{ and }
$$
%and 
$$
L_{\a\even\text{ at $\p$}} := \{\x \in L \mid H(\x, \x) \in \p^{\ord_\p(2\a)} \}.
$$
When $L$ is $\a$-valued these are actually {\bf sublattices}.  To see this, 
suppose 
$\x,\y \in L_{\a\even\text{(at $\p$)}}$ and notice that 
$$
H(\x + \y) = H(\x,\x) + H(\y,\y) + 2H(\x,\y) \subseteq 2\a \text{ (or $\p^{\ord_\p(2\a)}$)}
\implies 
\x + \y \in L_{\a\even\text{(at $\p$)}}.
$$
We say that $L$ is {\bf $\a$-even} if $L = L_{\a\even}$, and that $L$ is {\bf $\a$-even at $\p$} if $L = L_{\aevenatp}$.
\end{defn}

\begin{rem}[$\a$-even $\implies \a$-valued]
For a lattice $L$ in a bilinear space $(V,H)$ the property of being $\a$-even is stronger than being $\a$-valued.  
To see this, notice that if
$L$ is $\a$-even then %the polarization identity shows that 
$$H(\x,\y) = \frac{H(\x+\y,\x+\y) - H(\x,\x) - H(\y, \y)}{2} \in \frac{2\a}{2} = \a$$
for all $\x,\y \in L$, hence $L$ is $\a$-valued.
\end{rem}

\begin{rem}[Even observations]
Notice that the full and partial $\a$-even subsets/sublattices of an $\a$-valued bilinear lattice $L$ are related by the formula
$$L_{\a\even} = \cap_\p (L_{\a\even\text{ at $\p$}}),$$
and that $L$ is $\a$-even at all primes $\p\nmid 2$.
\end{rem}

For our purposes, the importance of $\a$-even bilinear lattices comes from the following simple correspondence with $\a$-valued quadratic lattices.

\begin{lem}[Even bilinear correspondence]  \label{lem:even_bilinear_correspondence}
Suppose that $L$ is a lattice in a 
%non-degenerate 
quadratic space $(V,Q)$, $\a$ is a non-zero fractional ideal of a ($\p$-adic or number) field $F$, and $(V,H)$ is the  
%(non-degenerate) 
Hessian bilinear space associated to $(V,Q)$.  Then 
$$
\begin{matrix} \text{$L \subseteq (V, H)$ is an} \\ \text{$\a$-even bilinear lattice} \end{matrix} 
\quad
\Longleftrightarrow 
\quad
\begin{matrix} \text{$L \subseteq (V, Q)$ is an} \\ \text{$\a$-valued quadratic lattice.} \end{matrix} 
$$
\end{lem}

\begin{proof}
This is just a restatement of the formula $H(\x,\x) = 2Q(\x)$.
\end{proof}

\begin{lem}[The index of an even sublattice] \label{lem:index_of_an_even_sublattice_when_two_is_unramified}
Suppose that $L$ is an $\a$-valued lattice  in a non-degenerate bilinear space $(V,H)$ over a number field $F$, 
for some fractional ideal $\a$ of $F$.  If $L$ is not $\a$-even at $\p$ and $e_\p := \ord_\p(2) \leq 1$, then  the quotient 
$$
L/L_{\a\even} \cong \F_\p
$$
as abelian groups.
\end{lem}

\begin{proof}
Notice that $\p\mid 2$ since otherwise $L_{\a\even} = L$, which is not true by assumption.  Consider the map $Q_H(\x) := H(\x, \x)$ which descends to a well-defined injective map 
$$
Q_H: L/L_{\a\even} \hookrightarrow \a/2\a. 
$$
From the definition of $L_{\a\even}$ and
the polarization identity we know that $Q_H$ 
is an (additive) homomorphism of abelian groups, but usually not a linear map of $\O_F/\p^{e_\p}$-modules.  
Since $e_\p \leq 1$ and $\p\mid 2$ we know that $e_\p = 1$ and so both the domain and range are $\F_\p$-vector spaces.  By Remark \ref{rem:bilinear_and_quadratic_in_char_2} we know that $Q_H$ is a semilinear map, and so its image is a non-zero subspace of $\F_\p$.  This shows $Q_H$ is surjective, hence bijective, proving the lemma.
\end{proof}

\begin{alg}[Construct a maximal $\a$-valued quadratic lattice] \label{alg:construct_maximal_quadratic_lattice}
Given a non-degenerate quadratic space $(V,Q)$ over a number field $F$ where $p=2$ is unramified, and a non-zero fractional ideal $\a$ of $F$, we give an algorithm for producing a maximal $\a$-valued quadratic lattice $L$ on $(V,Q)$.
\begin{enumerate}
\item Use Algorithm \ref{alg:construct_maximal_bilinear_lattice} to find an $\a$-maximal lattice $L$ in the Hessian bilinear space $(V,H)$ associated to $(V,Q)$.  
Let $\SS$ be the finite set of primes $\p$ of $\O_F$ where 
$L$ is not $\a$-even at $\p$.  (Necessarily $\p\in\SS \implies \p\mid 2$.)
\item For each $\p\in \SS$, we use Theorem \ref{thm:neighbors_as_nonsingular_points_and_construction} to check 
if some $\p$-neighbor $L'$of $L$ is $\a$-even at $\p$.
If so, then replace $L$ by $L'$.  If not, then replace $L$ by $L_{\aevenatp}$.
\item This remaining $L$ is a maximal $\a$-even lattice in $(V,H)$, so $L$ is also a maximal $\a$-valued quadratic lattice on $(V,Q)$.
\end{enumerate}
\end{alg}

\begin{proof}
Given $L$ in step 1, we know by the correspondence in Lemma \ref{lem:even_bilinear_correspondence} that $L$ is an $\a$-maximal quadratic lattice at all primes $\p\notin \SS$.  For each $\p\in \SS$ (so necessarily $\p\mid 2$) we know that $L_{\aevenatp}$ is $\a$-even at $\p$, but it may not be maximal among lattices in $(V,H)$ with this property.  If $L_{\aevenatp}$ is maximal $\a$-even at $\p$, then by Lemma \ref{lem:even_bilinear_correspondence} we see that it is also a maximal $\a$-valued {\it quadratic} lattice at $\p$ and at all primes $\q\notin \SS$.  If $L_{\aevenatp}$ is not maximal $\a$-even at $\p$, then we have the inclusion $L_{\aevenatp} \subsetneq L' \subseteq L'_{Max}$  where $L'$ maximal $\a$-even at $\p$, $L'_{Max}$ maximal $\a$-valued at $\p$,   
and $(L')_\q = (L'_{Max})_\q = L_\q$ for all primes $\q\neq \p$.

By the Lemma \ref{Lem:sublattice_discriminants}, we have that  
$$
[L:L_{\aevenatp}]^2 \cdot |\D(L)|
= |\D(L_{\aevenatp})|
= [L'_{Max}:L_{\aevenatp}]^2 \cdot |\D(L'_{Max})|,
$$
but Lemma \ref{lem:uniqueness_of_maximal_discriminant_modules} shows that $|\D(L)| = |\D(L'_{Max})|$, so $[L:L_{\aevenatp}] = [L'_{Max}:L_{\aevenatp}]$.  This together with Lemma \ref{lem:index_of_an_even_sublattice_when_two_is_unramified} shows that $[L'_{Max}:L_{\aevenatp}] = [L:L_{\aevenatp}] = \NormF(\p)$,  so $[L':L_{\aevenatp}]$ divides $\NormF(\p)$.

Since
$L'/L_{\aevenatp}$ is a non-zero torsion $\O_\p$-module its annihilator is $\p^k$ for some $k$, we know that $[L':L_{\aevenatp}] = \NormF(\p^k)$, so $k=1$ and $L' = L'_{Max}$.  
Finally, if $H(L, L') \subseteq \a$ then the lattice $L+L'$ would be $\a$-valued, violating the maximality of $L$.  Therefore $L'$ is a $\p$-neighbor of $L$.
\end{proof}

\begin{rem}[Maximal Quadratic lattices when $p=2$ is ramified]
If the prime $p=2$ ramifies in $F$ then the theory of integral quadratic forms/lattices is more complicated at any prime ideal $\p$ of $\O_F$ where $e_\p := \ord_\p(2) >1$. (E.g.  Lemma \ref{lem:index_of_an_even_sublattice_when_two_is_unramified}.)  In the language of O'Meara's book \cite[\S82E, p227]{OMeara:1971zr} this is because the containment of norm and (Gram) scale ideals
$$
2\s(L_\p) \subseteq \n(L_\p) \subseteq \s(L_\p)
$$
can 
both be strict containments 
when $e_\p> 1$, so one cannot normalize the values of quadratic form (given by $\n(L)$) by studying the values of the associated Hessian bilinear form (given by $2\s(L)$) alone.  In general there are many intermediate possibilities for the local norm ideals $\n(L_\p)$ which must be analyzed to pass from a maximal bilinear lattice to a maximal quadratic lattice, which makes it difficult to generalize this algorithm to deal with ramified primes $\p\mid 2$.  In a future paper, we hope to give a different algorithm that 
produces $\a$-maximal quadratic lattices over number fields where the prime $p=2$ is allowed to ramify. 
\end{rem}

%%%%%%%%%%%%%%%%%%%%%%
%%  SECTION: Neighbors and Genera  %%
%%%%%%%%%%%%%%%%%%%%%%
%\newpage
%%%%%%%%%%%%%%%%%%%%%%
%%  SECTION: Neighbors and Genera  %%
%%%%%%%%%%%%%%%%%%%%%%
\section{Neighbors and Genera}

In this section we will explain theory of ``Neighboring lattices", originally due to Kneser \cite{Kneser:1957fk}, and how it can be used to find representatives for a given genus of quadratic lattices.  Special care will be given to describe $\p$-neighbors for an arbitrary (possibly dyadic) prime ideal $\p$ in the ring of integers of a number field, as this will be useful in the passage from a maximal bilinear lattice to a maximal quadratic lattice.

%% Definition of $\p$-neighbors
\begin{defn}[$\p$-neighbors]
Suppose that $L$ and $L'$ are two $\a$-valued quadratic lattices in a common quadratic space $(V,Q)$ over a number field $F$, and that $\p$ is a (non-zero) prime ideal of $\O_F$.
Then we say that $L$ and $L'$ are {\bf $\p$-neighboring $\a$-lattices} (or {\bf $\p$-neighbors}) if $[L:L\cap L'] = [L':L\cap L']=\NormF(\p)$ and the Hessian bilinear form $H(L,L') \nsubseteq \a$.
\end{defn}

The $\p$-neighboring lattices of a given lattice $L$ can be described very explicitly in terms of the vectors of $L$ whose values generate the ideal $\a\p$.  These vectors also have a very nice description in terms of the residual quadric in the projective space of the $\F_\p$-vector space $L/\p L$, 
which we now define.

\begin{defn}[Residual Quadrics]
Suppose that $L$ is an $\a$-valued lattice in a $n$-dimensional quadratic space $(V,Q)$ over a number field $F$, and that $\p$ is a (non-zero) prime ideal of $\O_F$.  Then we define the {\bf residual $\a$-quadric at $\p$} as the quadric hypersurface $\C_{\p; \a} \subseteq \P(L/\p L) \cong \P^{n-1}(\F_\p)$  given by 
the homogeneous equation $Q(\x) \in \a\p$. I.e.,
$$
\C_{\p; \a} := \C_{L, \p; \a} := \P(\{ \x\in L/\p L %- \{\vec 0\} 
\mid Q(\x) \in \a\p  \text{ and } \x \neq \vec 0 \}).
$$
\end{defn}

\begin{rem}
In terms of the latttice $L$, those $\x \in L$ reducing to $\C_{\p; \a}$ are exactly those $\x \not\in \p L$ for which $Q(\x) \in \a\p$.
\end{rem}

The singular points on $\C_{\p; \a}$, which can be understood from several different perspectives.  Saying that $P\in\C_{\p; \a}$ is a {\bf nonsingular point} by definition means that the gradient vector $(\vec \nabla Q)(P) := [\frac{\partial Q}{\partial x_1}(P):  \cdots: \frac{\partial Q}{\partial x_n}(P)] \neq \vec 0 \in (\a/\p\a)^n$, where $[x_1 : \cdots : x_n]$ are homogeneous coordinates for $\P^{n-1}(\F_\p)$.  
We can also rewrite the gradient condition in terms of the Hessian bilinear form.

\begin{lem}[Hessian singularity criterion] \label{Lem:Hessian_and_singular_points}
Suppose that $P \in \C_{\p; \a}$ corresponds to the non-zero line $\F_\p \cdot \v_P \in L/\p L$.  Then $P$ is a singular point of $\C_{\p; \a} \iff$ the $(\a/\a\p)$-valued linear form $H(\v_P, \cdot)$  on $L/\p L$ is identically zero.  
\end{lem}
\begin{proof}
This follows from noticing that $H(\x, \y)  = (\vec \nabla Q)(\x) \cdot \y$, 
hence $P$ is singular $\iff H(\v_P, \w) \in \a\p$.
\end{proof}

\begin{rem}[Singular points and duality]
Another description of the singular points of $\C_{\p; \a}$ in terms of dual lattices can be given by noticing that for $\x,\y \in L$ we have that 
$$
H(\x, \y) \in \a\p \text{ for all }\y \in L \iff  \x \in (\p L^{\#\a} \cap L).
$$
This shows that the singular points of $\C_{\p; \a}$ are exactly the points of $\C_{\p; \a}$ lying in the projective subspace $\P((\p L^{\#\a} \cap L)/\p L) \subseteq \P(L/\p L)$.  For this reason, we call $\P((\p L^{\#\a} \cap L)/\p L)$ the {\bf (residual) $\a$-singular subspace}.

It is interesting to notice that the singular subspace is closely related to the structure of the $\a$-discriminant module $\D_\a(L)$, and that for $\a$-modular lattices (where $L^{\#\a} = L$) there are no singular points on $\C_{\p; \a}$.
\end{rem}

With these definitions we can now parametrize the $\p$-neighbors of $L$ in terms of the non-singular points of its residual quadric at $\p$.

\begin{thm}[$\p$-neighbors via residual geometry]  \label{thm:neighbors_as_nonsingular_points_and_construction}
The $\p$-neighboring $\a$-lattices of a given $\a$-valued 
quadratic lattice $L$ are in bijection with the non-singular points $\C_{\p; \a}^\text{ns}$ on its residual $\a$-quadric at $\p$.
More explicitly, this bijection is given by the map
$$
\eta: P \in \C_{\p; \a}^\text{ns} 
\quad 
\mapsto 
\quad 
L' := L'' + \tfrac{1}{\p} \cdot \v_P
$$
where $L'' := \{\x \in L \mid H(\v_P, \x) \equiv 0 \pmod {\a\p}\}$
 and $\v_P$ is any lift to $L$ of a non-zero isotropic vector in $L/\p^2 L$ (i.e. $Q(\v_P) \in \p^2 \a$ and $\v_P \notin \p^2L$)  that reduces to the given point $P\in\C_{\p; \a}^\text{ns}$ under the canonical reduction map $L \ra \P(L/\p L)$.
Such a $\v_P \in L$ always exists since $P$ is a non-singular point,
 and also $L'' = L \,\cap\, L'$.
\end{thm}

\begin{proof}
{\it 1) Non-singular shearing:}
To see that $P \in \C_{\p; \a}^\text{ns}$ gives rise to some $\v_P$ as above, notice that Lemma \ref{Lem:Hessian_and_singular_points} says that the gradient $(\vec \nabla Q)(\v_P) \in L/\p L$ is non-zero in $\a/\p\a$, so by chosing any lift $\w \in L$ of $(\vec \nabla Q)(\v_P)$ we can find some $\lambda \in \p$ so that $Q(\v_P + \lambda \w) \in \a\p^2$.

\smallskip
{\it 2) $\p$-neighbors:}
To see that $\eta( P)$ is a $\p$-neighboring $\a$-lattice of $L$, we notice that by Lemma \ref{Lem:Hessian_and_singular_points} the non-singularity of $P$ shows that the linear form $H(\v_P, \cdot)$ on $L$ has non-trivial image in $\a/\p\a$, hence $L/L'' \cong \F_\p$ as abelian groups and $[L:L''] = \NormF(\p)$.  Since $\v_P \notin \p L$, we know $\p^{-1} \v_P \notin L$ and so similarly $L/L'' \cong \F_\p$ and  $[L':L'']= \NormF(\p)$.  
Finally we see that $L$ and $L'$ are $\p$-neighbors since from 1) we have $H(\v_P, \w) \notin \p\a$ for some $\w \in L$, so therefore 
$H(\w, \p^{-1}\v_P)\subseteq H(L, L') \not\subseteq \a$.

\smallskip
{\it 3) Injectivity:}
To see that $\eta( P)$ is injective suppose that $\v_P$ and $\v_Q$  in $L$ correspond to the points $P$ and $Q$ in $\C_{\p; \a}^\text{ns}$, and that $\eta( P) = \eta(Q) = L'$.   Then from 2) we know that $L + L' = \p^{-1}\v_P + L = \p^{-1}\v_Q + L$, so $\v_P$ and $\v_Q$ lie on the same residual line through the origin $\p(L + L')/\p L \subseteq L/\p L$, so $P = Q$ on $\C_{\p; \a}^\text{ns} \subseteq \P(L/\p L)$.

\smallskip
{\it 4) Surjectivity:}
To see that $\eta( P)$ is surjective we first start with a $\p$-neighbor $L'$ of $L$ and construct some point $P\in\C_{\p; \a}^\text{ns}$ by the rule 
$
P := \P(\p(L+ L')/\p L) \subset \P(L/\p L).
$
Since $[L + L':L] = [L': L\cap L'] = \NormF(\p)$ we know that $\dim_{\F_\p}(L + L') = \dim_{\F_\p}(L) + 1$ and so $P$ is a point in $\P(L/\p L)$.  To see that $P \in \C_{\p; \a}$ we choose some non-zero $\v' \in L'/(L \cap L')$ giving $L + L' = L + \O_F \v'$ and set $\v_P := \pi_\p \v'$ (giving $\P(\v_P) = P$).  Knowing $\v' \in L'$ ensures that $Q(\v') \in \a$, so $Q(\v_P) \in \a\p^2$ showing $P \in \C_{\p; \a}$.  The non-singularity of $P$ follows 
because 
$$
H(\p(L + L'), L') = H(\p L, L) + H(\p L', L) \subseteq \p\a + \a \subseteq \a
$$
(using here that $\p L' \subseteq L \implies H(L, L')\subseteq \p^{-1}\a$), 
and if $P$ were singular then it would force $H(L', L) \subseteq \a$, which cannot occur so $P \in \C_{\p; \a}^\text{ns}$.

To see that $\eta( P) = L'$, we first compute the sublattice
$$
K := \{ \x \in L + L' \mid H(\v', \x) \subseteq \a\}  \subset L + L'.
$$
for $\v'$ as above, which is proper since $H(L,L') \not\subseteq \a$.  Since $\v' \in L'$ we see that $L' \subseteq K$, and $H(L',L') \subseteq \p^{-1}\a$ tells us that $(L+L') / K\cong \F_\p$ as abelian groups, giving $[L+L' : K]= \NormF(\p)$ and so $K= L'$.  With this, we see that $\eta$ first constructs the sublattice $L'' = L \cap K = L \cap L'$ and then takes $L' := \p^{-1}\v_P + L''$, and this $L'$ must  be the $\p$-neighbor $L'$ that we started with since $\pi_\p^{-1} \v_P = \v' \in L'$.  This completes the proof the Theorem.
\end{proof}

\begin{rem}[$\p$-neighbor references]
The idea of neighboring lattices was introduced in Kneser's paper \cite{Kneser:1957fk}, and is also discussed in his German book \cite{Kneser_book}.
There are other discussions of $\p$-neighboring lattices in English (e.g. \cite[p202]{Gerstein_book}, \cite[\S3.1, pp31-35]{tornaria_thesis}, \cite[\S1, pp1-3]{Schulze-Pillot:1991aw}, \cite[\S2, pp738--743]{Scharlau-Hemkemeier}), though 
these restrict either the base field (i.e. $F = \Q$) or the primes considered (i.e. $\p\nmid |\D(L)|$) or both, and also assume that the quadratic lattices are $\O_F$-valued.  

The author has been unable to find a description of 
$\p$-neighbors for $\a$-maximal lattices discussed in the literature, so the notion of $\p$-neighboring $\a$-lattices here appears to be new (though perhaps not very deep).  We include it here 
to show that there is a natural notion of $\p$-neighbors for lattices of any fixed value ideal $\a:= Q(L) \cdot \O_F$,
and to provide context for results like \cite{MR0337978, Brzezinski1974} that begin to describe connections between arithmetic and geometric models of lattices.
This connection is an interesting direction for future research.
\end{rem}

For completeness, we conclude with an important application of the theory of $\p$-neighboring lattices 
to enumerate all classes in a given genus of totally definite quadratic lattices over a totally real number field $F$. 
This result is essentially due to Benham and Hsia \cite{Benham_Hsia} where they actually give a way to compute all classes in the spinor genus of $L$, and here we describe a well-known modification of their idea using a mass formula to determine when 
the algorithm terminates.

The advantage of this modification is that we do not need to compute spinor norms in the idele group of $F$, but it comes at the rather high cost of needing an exact mass formula for the genus of quadratic lattices in question.

\begin{alg}[Enumerating classes in a genus of quadratic lattices; \cite{Benham_Hsia}]  \label{alg:enumerate_classes_of_quadratic_lattices_in_a_genus}
Given an $\a$-valued quadratic lattice $L$ in a non-degenerate totally definite quadratic space $(V,Q)$ over a totally real number field $F$ of rank $\geq 3$, then we give an algorithm for finding representative lattices $L_i \subset V$ for every class in the genus of $L$.  
\end{alg}

\begin{proof}
Begin with the set of lattices $\SS = \{L\}$. Take the smallest (w.r.t. $|\NormF(\cdot)|$) prime $\p\nmid |\D(L)|$ and compute the set $\mathbb{T}$ of (finitely many) non-isometric $\p$-power neighbors of each $L \in \SS$ and append $\mathbb{T}$ to $\SS$.  
(Here we know we have computed $\mathbb{T}$ when taking $\p$-neighbors of all classes of lattices of $\mathbb{T}$ produces no new classes.)
If the partial mass 
$$
\textrm{Mass}(\SS) := \sum_{L' \in \SS}\frac{1}{\#\mathrm{Aut}(L')}
$$
satisfies $\mathrm{Mass}(\SS) < \mathrm{Mass}(L)$, then repeat the procedure for the next smallest prime until $\textrm{Mass}(\SS) = \textrm{Mass}(L)$.
By \cite[Proposition 1, (1.1), and Theorem 2]{Benham_Hsia} the $\p$-power neighbors at the primes $\p$ of bounded norm 
%will 
exhaust all classes in the genus of $L$.
%, hence there must be finitely many that are non-isometric for any $\p\nmid |\D(L)|$.
\end{proof}

\begin{rem}[Mass formula and halting conditions for indefinite lattices] \label{rem:halting_conditions_for_indefinite_lattices}
Algorithm \ref{alg:enumerate_classes_of_quadratic_lattices_in_a_genus} also works for indefinite lattices $L$, though 
%here 
in this case 
the mass of a genus $\mathrm{Gen}(L)$ is given as a sum over all class representatives $L_i$ of the covolumes $\mathrm{Vol}(\mathcal{Z}/\mathrm{Aut}_{\O_F}(L_i))$ of the integral automorphism group $\mathrm{Aut}_{\O_F}(L_i)$ with respect to some fixed measure on the symmetric space $\mathcal Z$ of the orthogonal group of $Q$ (e.g. \cite{Ha_thesis_paper}).  While these terms are computable in principle (by giving a presentation for $\mathrm{Aut}_{\O_F}(L_i)$ and computing an explicit integral), this is not nearly as pleasant as counting the finitely many automorphisms arising in the totally definite case.
\end{rem}

\begin{rem}[Class numbers for indefinite lattices when $n\geq 3$]
Conveniently, the computation of class numbers of indefinite lattices in $n \geq 3$ variables is much simpler that for definite lattices because of the strong approximation property of the spin group (e.g. \cite[\S104:4-5, pp315-319; \S102:7-8, pp300-304]{OMeara:1971zr}).  From this property one can show that each spinor genus in such a genus contains exactly one class, so the class number is just the number of spinor genera in the genus and this number is known to be an easily (locally) computed power of two.  Therefore Algorithm \ref{alg:enumerate_classes_of_quadratic_lattices_in_a_genus} and Remark \ref{rem:halting_conditions_for_indefinite_lattices} would only be interesting when $n=2$, in which case these masses can be computed in terms of logarithms of fundamental units in quadratic extensions of $F$.
\end{rem}

\begin{rem}[Literature]  \label{rem:class_enumeration_literature}
The idea of using an explicit mass formula to determine explicit genus representatives is well-known to experts, 
and has been used for some time to prove that certain genera have class number one 
(e.g. \cite[\S16.6, pp133-134]{Siegel:1963vn}, \cite[\S5.16, pp33-34]{Shimura_exact_mass}) or small class number, 
however almost all actual computational results have been limited to either the case $F= \Q$ or to unimodular lattices over real quadratic fields.  
When $[F:\Q] > 1$, the author is only aware of the papers \cite{Costello:1987kx,Hsia:1989yq,Hsia:1989vn,Hung:1991ys,Scharlau:1994ys}.
\end{rem}

\bibliographystyle{plain}
\bibliography{banff_maximal}

\end{document}